\documentclass[12pt,reqno]{amsart}

\usepackage[T1]{fontenc}
\usepackage[utf8]{inputenc}
\usepackage{amsmath,amssymb,mathtools}
\usepackage[margin=1.25in]{geometry}
\usepackage{microtype}

\usepackage{comment}
\usepackage[
    colorlinks=true,
    linkcolor=black,
    citecolor=black,
    urlcolor=blue
]{hyperref}

\newtheorem{theorem}{Theorem}[section]
\newtheorem{proposition}[theorem]{Proposition}
\newtheorem{lemma}[theorem]{Lemma}

\theoremstyle{definition}
\newtheorem{definition}[theorem]{Definition}
\theoremstyle{remark}

\numberwithin{equation}{section}

\newcommand{\R}{\mathbb{R}}
\newcommand{\Sym}{\operatorname{Sym}}
\newcommand{\tr}{\operatorname{tr}}
\newcommand{\Arg}{\operatorname{Arg}}
\newcommand{\sgn}{\operatorname{sgn}}
\newcommand{\1}{\mathbf{1}}
\newcommand{\cA}{c_{\alpha}}
\newcommand{\phase}{\mathcal{F}}

\allowdisplaybreaks
\title[Nonuniqueness for Lagrangian mean curvature equation]
{Nonuniqueness of Solutions to the Lagrangian Mean Curvature Equation}

\author{Arunima Bhattacharya}
\address{Department of Mathematics, Phillips Hall, The University of North Carolina at Chapel Hill, Chapel Hill, NC 27599}
\email{arunimab@unc.edu}

\author{W. Jacob Ogden}
\address{Department of Mathematics, Mathematics Hall, Columbia University, New York, NY 10027}
\email{wjo2114@columbia.edu}

\dedicatory{(Dedicated to F. Reese Harvey and H. Blaine Lawson, Jr.,
on the occasion of their 85th birthdays)}

\begin{document}

\begin{abstract}
We resolve a question posed by Harvey and Lawson concerning the uniqueness of continuous viscosity solutions to the Dirichlet problem for the Lagrangian mean curvature equation on the unit ball with continuous boundary data.
 For every dimension $n\geq 2$, we construct a continuous phase
$\psi\colon \overline{B_1}\to(-n\frac{\pi}{2},n\frac{\pi}{2})$ and continuous boundary data
$g\colon \partial B_1\to\R$ for which the Dirichlet problem
\[
 \sum_{i=1}^n \arctan\lambda_i(D^2u)=\psi(x)\quad\text{in }B_1,
 \qquad u=g\quad\text{on }\partial B_1,
\]
admits a continuum of distinct continuous viscosity solutions.   
\end{abstract}

\maketitle

\section{Introduction}

Let $\Omega\subset\R^n$ be a bounded domain.  We consider the Dirichlet problem
for the Lagrangian mean curvature equation
\begin{equation}\label{DP}
 \begin{cases}
 \displaystyle \phase_n(D^2u):=\sum_{i=1}^n\arctan\lambda_i(D^2u)=\psi(x)
 &\text{in }\Omega,\\[4pt]
 u=g&\text{on }\partial\Omega,
 \end{cases}
\end{equation}
where $\lambda_1(D^2u),\dots,\lambda_n(D^2u)$ are the eigenvalues of the
Hessian, $\psi$ is the prescribed Lagrangian phase, and $g$ is the boundary
value.  Our main result gives a negative answer to the uniqueness question for
continuous viscosity solutions in every dimension $n\geq 2$. In dimensions
$n\geq 3$, the phase can be chosen to cross any of the special values $-(n-2m) \frac \pi 2$, where $m\in\{1,\dots,n-1\}$, at the origin.

\begin{theorem}\label{thm:DP}
For every $n\geq2$, there exist
\[
 \psi_n\in C\big(\overline{B_1};(-n\frac{\pi}{2},n\frac{\pi}{2})\big),
 \qquad g_n\in C(\partial B_1),
\]
such that \eqref{DP}, with $\Omega=B_1$, admits a continuum of
distinct solutions in $C(\overline{B_1})$ in the viscosity sense.

Moreover, for every $n\geq3$ and every $m\in\{1,\dots,n-1\}$, the phase $\psi_n$ and boundary value $g_n$ may be chosen so that
\begin{equation}\label{CriticalCross}
 \psi_{n}(0)=-(n-2m)\frac{\pi}{2},
 \qquad
 \psi_{n}(te_1)<-(n-2m)\frac{\pi}{2}<\psi_{n}(te_n)
\end{equation}
for all sufficiently small $t>0$.  In particular, $\psi_{n}$ crosses the prescribed special value.
\end{theorem}

When the phase, $\psi$, is constant, \eqref{DP} is the special
Lagrangian equation introduced by Harvey-Lawson in their theory
of calibrated geometries \cite{HL82}. 
For classical solutions $u$, constant phase implies the gradient graph 
\[
 \Gamma_u=\{(x,Du(x)):x\in\Omega\}\subset\R^n\times\R^n
\]
is a calibrated, hence volume-minimizing, submanifold of $\R^n \times \R^n$. More generally, when $u$ is a smooth solution with variable phase $\psi$, the Lagrangian graph $\Gamma_u$ has Lagrangian angle or phase $\psi$
and mean curvature 
$\vec H=J\nabla_g\psi$, where 
$g = I + (D^2u )^2 $ is the induced metric and $J$ is the standard complex structure on $\mathbb{C}^n = \R^n \times \R^n$ \cite[III (2.19)]{HL82}.

The threshold $\pm(n-2)\frac{\pi}{2}$ is called the critical phase. It was shown by Yuan \cite{Yuan06} that the level sets of the arctangent operator are convex when the absolute value of the phase is larger than the critical value. When the absolute value of the phase lies strictly below this critical value, the arctangent operator loses all convexity properties. Also, note that the arctangent operator is not uniformly elliptic unless there is a uniform bound on the Hessian. 

In the critical and supercritical range, i.e., $|\psi|\geq (n-2)\frac{\pi}{2}$, regularity for the special Lagrangian equation has been extensively studied. In this range, in 
dimension two and three, Warren-Yuan obtained explicit interior Hessian and gradient
estimates 
\cite{WarrenYuan09,WarrenYuan10}.  Wang-Yuan extended the interior Hessian estimate to
critical and supercritical phases in every dimension \cite{WangYuan14}.  For
convex potentials, estimates and regularity were developed by Bao-Chen
\cite{BaoChen03}, Chen-Warren-Yuan \cite{ChenWarrenYuan09}, and Chen-Shankar-Yuan \cite{ChenShankarYuan23}.  Further approaches include Li's
compactness argument \cite{Li19} and Shankar's doubling proof
\cite{Shankar26}.  Constant-rank and Liouville results for saddle solutions beyond the supercritical range
under a semiconvexity hypothesis were obtained by Ogden-Yuan
\cite{OgdenYuan25}.

When the phase is variable, substantially more information about the phase is
needed.  Warren obtained Hessian estimates for convex smooth solutions with
$C^{1,1}$ phase \cite{Warren08}.  For critical and supercritical $C^{1,1}$ phases,
Bhattacharya established interior Hessian estimates \cite{Bhattacharya21, Bhattacharya22}; Bhattacharya-Mooney-Shankar proved the complementary gradient estimate for $C^2$ phases
\cite{BhattacharyaMooneyShankar}.  Together, these estimates yield interior $C^3$
regularity and solvability of the Dirichlet problem for \eqref{DP} with $C^{1,1}$ phase and continuous boundary
data on a uniformly convex bounded domain \cite{BhattacharyaMooneyShankar, Bhattacharya24}; see also Lu \cite{Lu23}. For convex viscosity
solutions, interior regularity and sharp regularity criteria were obtained by
Bhattacharya-Shankar \cite{BhattacharyaShankar23,BhattacharyaShankar24}.
Recent work of Zhou \cite{Zhou25} and Ding \cite{Ding24} establishes Hessian estimates under a Lipschitz phase assumption when the phase is strictly or weakly supercritical, respectively, but the corresponding gradient estimates remain open.

The restrictions on the phase range above are natural for regularity theory.  In the subcritical range, i.e., $|\psi|< (n-2)\frac{\pi}{2}$, Nadirashvili-Vl\u{a}du\c{t} \cite{NadirashviliVladut} constructed $C^{1,\frac{1}{3}}$ solutions in dimension three, while Wang-Yuan \cite{WangYuan13} proved the existence of solutions $u \in C^{1,\frac{1}{2m+1}} \setminus C^{1,\frac{1}{2m+1}+\varepsilon}$ for every integer $m \geq 1$ in dimension three.  Mooney-Savin later produced Lipschitz viscosity
solutions that are not $C^1$ \cite{MooneySavin}, and Mooney-Shankar obtained
sharp regularity results for semiconvex viscosity solutions accompanied by singular examples
\cite{MooneyShankar25}.

In \cite{HL09}, Harvey-Lawson developed the Dirichlet theory for degenerate elliptic equations of the form $F(D^2u)=0$ and established comparison and uniqueness for the special Lagrangian equation, together with existence under the corresponding strict boundary convexity assumptions (see also \cite{YnotesSummer08} and \cite[Theorem 1.2]{Bhattacharya24}). In \cite{BrW}, Brendle-Warren studied the second boundary value problem for the special Lagrangian equation.

For variable $C^{1,1}$ phase in the supercritical
range, the Dirichlet problem was treated under subsolution existence
hypotheses by Collins-Picard-Wu \cite{CPW17} and in the continuous
viscosity setting by Dinew-Do-T{\^o} \cite{DinewDoTo}; see also
\cite{Bhattacharya24} for a direct treatment with continuous boundary data on a uniformly convex bounded domain.

Beyond the top phase interval, $\pm ((n-2)\frac{\pi}{2}, n\frac{\pi}{2})$, however, available comparison theory is
formulated interval by interval.  The dividing values, or the so-called special values, are
\begin{equation*}\label{eq:special-values}
 \theta_k=(n-2k)\frac{\pi}{2},\qquad k=1,\dots,n-1,
\end{equation*}
and the intervening phase intervals are
\begin{equation*} \label{intervals}
 I_k=\left((n-2k)\frac{\pi}{2},(n-2(k-1))\frac{\pi}{2}\right),
 \qquad k=1,\dots,n.
\end{equation*}
For $n\geq3$, the outer special values $\theta_1$ and $\theta_{n-1}$ are the
critical phases discussed above, while the intermediate special values lie
strictly in the subcritical range.  These values are distinguished by the
asymptotic behavior of the operator.  Indeed, if $k$ eigenvalues of a sequence
of Hessians tend to $-\infty$ and the remaining $n-k$ tend to $+\infty$, then
the corresponding phases converge to
$(n-2k)\frac{\pi}{2}=\theta_k$.  Thus the $\theta_k$ are precisely the values
approached by mixed-sign configurations in which every Hessian eigenvalue
diverges in magnitude.

In \cite{HL19}, Harvey-Lawson introduced a condition on the arctangent operator known as “tameness.” Slightly stronger than strict ellipticity, this condition is sufficient to establish comparison.
Building on the comparison theorem of Cirant-Payne \cite{CirantPayne}, which builds on a theory of Trudinger \cite{True}, Harvey-Lawson proved Dirichlet solvability in \cite{HL21}, under the corresponding strict boundary convexity hypothesis, whenever the range of $\psi$ is contained in one interval $I_k$. In the same paper, they explicitly identified the
extension to a general continuous phase, allowed to attain or cross the values
\eqref{eq:special-values}, as an open problem.

Brustad subsequently showed that the phase-interval hypothesis cannot, in
general, be omitted from comparison. For each special value $\theta_k$, Brustad constructed a phase function crossing $\theta_k$ and a viscosity subsolution and supersolution whose
difference has a strict interior maximum, and hence comparison fails
\cite{Brustad24}. This result does not disprove uniqueness for the Dirichlet
problem, since the subsolution and supersolution in the construction are not
solutions with the given boundary data. Theorem~\ref{thm:DP} establishes the
corresponding nonuniqueness statement: a single continuous phase and a
single continuous boundary datum may produce a continuum of distinct
continuous viscosity solutions. We also note that Cao-Wang constructed
infinitely many $C^{1,\alpha}$ solutions in dimension two by convex
integration \cite{CaoWang}. Their result is formulated in a different weak
solution class and does not address uniqueness in the viscosity setting.

The distinction between viscosity solutions and Cao-Wang's very weak solutions is important for the Dirichlet problem.  Comparison is the
mechanism behind uniqueness, stability under approximation, and Perron's
method for viscosity solutions \cite{CIL}.  The present examples show that,
once the phase crosses a special value, the Dirichlet problem is genuinely ill-posed in
the canonical continuous viscosity solution class.

\subsection*{Outline of the Proof}We briefly describe our construction.  In dimension two, several singular
interfaces must be balanced so that one Hessian eigenvalue tends to $+\infty$
and the other to $-\infty$ at the codimension-two junctions.  This keeps the
extended phase inside $(-\pi,\pi)$.  In dimensions $n\geq3$, a simpler model is
available.  We set
\[
 \rho(x)=\sum_{i=1}^n|x_i|,
 \qquad
 v(x)=\sum_{j=1}^m|x_j|^\alpha-\sum_{j=m+1}^n|x_j|^\alpha+|\rho(x)-R|^\alpha,
 \qquad 1<\alpha<2,
\]
and define $\psi=\phase_n(D^2v)$ away from the singular skeleton.  A rank-one
determinant formula shows that $\psi$ extends continuously across every
stratum.  We then add a one-parameter family of piecewise affine functions,
including a compactly supported tent $C(R-\rho)^+$.  These additions have zero
Hessian off the skeleton, and the $C$-dependent tent vanishes at the boundary; hence every member of the resulting
family has the same classical phase away from the singular set and the same boundary values.  On the singular set, convex and concave cusps either rule out
one or both kinds of test functions, while Cauchy interlacing verifies the
remaining viscosity inequalities.  Finally, an expansion of the phase along
the $x_1$- and $x_n$-axes gives \eqref{CriticalCross}.

\subsection*{Organization} In Section \ref{sec:prelim}, we present the viscosity convention and the elementary
linear algebra used throughout. In Section \ref{sec:two-dim}, we show
the two-dimensional construction.  In Section \ref{sec:higher-dim}, we prove the result
in every dimension $n\geq3$.

\subsection*{Acknowledgments} 
AB is especially grateful to Reese Harvey and Blaine Lawson for bringing the difficulty of this problem to her attention six years ago and for generously sharing their insights and notes with her, which clarified the obstacles to establishing a comparison principle. AB and WJO would like to thank Yu Yuan for helpful discussions. AB acknowledges the support of NSF grant DMS-2350290, Simons Foundation grant MPS-TSM-00002933, and a Bill Guthridge Fellowship from UNC-Chapel Hill.

\section{Preliminaries}\label{sec:prelim}
For $A\in\Sym^2(\R^m)$, let \[\phase_m(A)=\sum_{i=1}^m\arctan\lambda_i(A)\] and
 $\phase_0=0.$ The operator is continuous and strictly increasing in the matrix order.  We
recall the following viscosity sub- and supersolution definitions.

\begin{definition}\label{def:viscosity}
A function $u\in C(\Omega)$ is a viscosity subsolution of
$\phase_n(D^2w)=\psi$ if, whenever $\varphi\in C^2$ touches $u$ from above at
$x_0\in\Omega$, one has
\[
 \phase_n(D^2\varphi(x_0))\geq\psi(x_0).
\]
It is a viscosity supersolution if, whenever $\varphi$ touches $u$ from below
at $x_0$, one has
\[
 \phase_n(D^2\varphi(x_0))\leq\psi(x_0).
\]
If $u$ is both a viscosity subsolution and supersolution, then $u$ is a viscosity solution. 
\end{definition}

We first record the phase inequalities for compressions of a symmetric matrix.
If $T\subset\R^m$ is a linear subspace, $A|_T$ denotes the compression of $A$
to $T$.

\begin{lemma}\label{lem:compression}
Let $A\in\Sym^2(\R^m)$ and let $T\subset\R^m$ have codimension $k$.  Then
\begin{equation*}
 -k\frac{\pi}{2}+\phase_{m-k}(A|_T)
 \leq \phase_m(A)
 \leq k\frac{\pi}{2}+\phase_{m-k}(A|_T).
\end{equation*}
\end{lemma}

\begin{proof}
Let $\lambda_1\geq\cdots\geq\lambda_m$ be the eigenvalues of $A$ and let
$\mu_1\geq\cdots\geq\mu_{m-k}$ be those of $A|_T$.  Cauchy interlacing gives
$\lambda_j\geq\mu_j\geq\lambda_{j+k}$.  Since $\arctan$ is increasing and
takes values in $(-\frac{\pi}{2},\frac{\pi}{2})$,
\[
 \sum_{j=1}^m\arctan\lambda_j
 \leq k\frac{\pi}{2}+\sum_{j=1}^{m-k}\arctan\mu_j,
\]
and the lower inequality follows in the same way.
\end{proof}

The next identity is a rank-one phase formula, which is the main tool for extending the higher-dimensional phase.

\begin{lemma}\label{lem:rank-one-formula}
Let $D=\operatorname{diag}(d_1,\dots,d_m)$, let
$\sigma\in\{\pm1\}^m$, and let $b>0$.  Define
\[
 P=\sum_{j=1}^m\frac{1}{1+d_j^2},
 \qquad
 Q=\sum_{j=1}^m\frac{d_j}{1+d_j^2}.
\]
Then
\begin{equation}\label{eq:rank-one-phase}
 \phase_m(D+b\,\sigma\otimes\sigma)
 =\sum_{j=1}^m\arctan d_j
 +\Arg(1+bQ+i bP),
\end{equation}
where the argument of the second term is the value in $(0,\pi)$.
\end{lemma}

\begin{proof}
The matrix determinant lemma gives
\begin{align*}
 \det\big(I+i(D+b\,\sigma\otimes\sigma)\big)
 &=\det(I+iD)\left(1+i b\,\sigma^T(I+iD)^{-1}\sigma\right)\\
 &=\prod_{j=1}^m(1+i d_j)\,(1+bQ+i bP).
\end{align*}
The second factor lies in the open upper half-plane.  To fix the branch of the
argument, consider $D+t b\,\sigma\otimes\sigma$, $0\leq t\leq1$.  Its phase is
continuous, starts at $\sum_j\arctan d_j$, and its derivative is
\[
 b\,\tr\!\left((I+(D+t b\,\sigma\otimes\sigma)^2)^{-1}
                  \sigma\otimes\sigma\right)>0.
\]
The argument of $1+t bQ+i t bP$ is the continuous branch beginning at zero,
so no multiple of $2\pi$ occurs in \eqref{eq:rank-one-phase}.
\end{proof}

We will also use the standard rank-one blow-up limit.

\begin{lemma}\label{lem:rank-one-limit}
Let $B\in\Sym^2(\R^m)$ and let $\nu$ be a unit vector.  Then
\begin{equation}\label{eq:rank-one-limit}
 \lim_{t\to+\infty}\phase_m(B+t\,\nu\otimes\nu)
 =\frac{\pi}{2}+\phase_{m-1}(B|_{\nu^\perp}),
\end{equation}
and
\begin{equation}\label{eq:rank-one-limit-negative}
 \lim_{t\to+\infty}\phase_m(B-t\,\nu\otimes\nu)
 =-\frac{\pi}{2}+\phase_{m-1}(B|_{\nu^\perp}).
\end{equation}
\end{lemma}

\begin{proof}
After a rotation, take $\nu=e_1$.  One eigenvalue of
$B+t e_1\otimes e_1$ tends to $+\infty$, while its remaining $m-1$
eigenvalues converge to the eigenvalues of the compression $B|_{e_1^\perp}$.
For example, the latter assertion follows by dividing the characteristic
polynomial by $t$ and passing to the limit on bounded sets.  Applying
$\arctan$ gives \eqref{eq:rank-one-limit}; the negative limit is identical.
\end{proof}

\section{The two-dimensional construction}\label{sec:two-dim}
In this section, we present the two-dimensional construction. Let $\rho (x) = |x_1| + |x_2|$. Fix $1 < \beta < \alpha < 2 $ and $0<R < 1 $,  and let 
$$ v(x) = |x_1|^\alpha - | x_2|^\alpha + | \rho ( x) - R | ^\alpha - |x_2^2-R^2|^\beta -| x_2^2 - R^2 |. $$ 
Note $v$ is $C^2$ away from the set $$S=
\{x_1=0\}
\cup
\{x_2=0\}
\cup
\{\rho(x)=R\}
\cup
\{x_2=\pm R\}.$$
On $\overline{B_1} \setminus S$, define 
$$ \psi(x) = \sum_{i = 1}^2 \arctan (\lambda_i( D^2 v (x))) .$$

\begin{proposition}
\label{lem:2d-phase} The function $\psi : \overline{B_1} \setminus S \to \mathbb{R} $ extends to a continuous function on $\overline B_1$ and takes values in $( - \pi , \pi ). $
\end{proposition} 

\begin{proof}
Recall that for a $2\times2$ symmetric matrix $A$, with eigenvalues
$\lambda_1(A)\geq\lambda_2(A)$,
\[
 \lambda_1(A)=\sup_{e\in S^1}\langle e,Ae\rangle,
 \qquad
 \lambda_2(A)=\inf_{e\in S^1}\langle e,Ae\rangle.
\]
Away from $S$, we have
\begin{align}
    v_1 & = \alpha \sgn x_1 ( |x_1|^{\alpha-1} + | \rho (x) - R | ^{\alpha -1 } \sgn ( \rho (x) - R ) ),\notag \\ 
    v_2 & = \alpha \sgn x_2 ( -|x_2|^{\alpha-1} + | \rho (x) - R | ^{\alpha -1 } \sgn ( \rho (x) - R ) ) \notag
    \\ & \quad - 2 x_2 \sgn ( x_2^2 - R^2 ) ( \beta | x_2^2 - R^2 |^{\beta -1} + 1 )   ,\notag \\
    v_{11} & = \alpha( \alpha - 1) (| x_1 | ^{\alpha-2} + | \rho(x) - R|^{\alpha -2} ), \label{eq:v11}\\
    v_{12} & = \alpha ( \alpha -1 ) | \rho ( x) - R| ^{\alpha -2 } \sgn x_1 \sgn x_2 , \label{eq:v12}\\
    v_{22} & = \alpha ( \alpha -1 ) ( - | x_2 |^{\alpha -2 } + | \rho ( x) - R | ^{\alpha -2} ) - 4 \beta (\beta -1 ) | x_2^2 - R^2 |^{\beta -2} x_2^2 \notag \\ & \quad - 2 \sgn ( x^2_2 - R^2 ) ( \beta | x_2^2 - R^2 |^{\beta -1 } +1). \label{eq:v22}
\end{align}
We say a point $x_0$ of $S$ is regular if $S$ is a smooth manifold in a neighborhood of $x_0$.
We examine the regular pieces of $S$ first.  At a regular point $(0,y)$ on the piece
$\{x_1=0\}$, the only unbounded term in the $e_1$ direction is
$\alpha(\alpha-1)|x_1|^{\alpha-2}$.  The elementary eigenvalue formula for a
$2\times2$ symmetric matrix therefore gives
\[
 \lambda_1(D^2v)\longrightarrow+\infty,
 \qquad
 \lambda_2(D^2v)\longrightarrow v_{22}(0,y).
\]
At a regular point $(y,0)$ on $\{x_2=0\}$, the analogous negative blow-up gives
\[
 \lambda_2(D^2v)\longrightarrow-\infty,
 \qquad
 \lambda_1(D^2v)\longrightarrow v_{11}(y,0).
\]
At a regular point on $\{\rho=R\}$, the singular part of the Hessian is a
positive rank-one term in the normal direction to $S$. Lemma~\ref{lem:rank-one-limit}
then gives
\[
 \lambda_1(D^2v)\longrightarrow+\infty,
 \qquad
 \lambda_2(D^2v)\longrightarrow v_{\tau\tau},
\]
where $\tau$ is a unit tangent vector to $\{\rho=R\}$.  Finally, at a regular
point of $\{x_2=\pm R\}$, the term
\[
 -4\beta(\beta-1)|x_2^2-R^2|^{\beta-2}x_2^2
\]
in \eqref{eq:v22} tends to $-\infty$, while the other entries relevant to the
tangential compression have finite limits.  Hence
\[
 \lambda_2(D^2v)\longrightarrow-\infty,
 \qquad
 \lambda_1(D^2v)\longrightarrow v_{11}.
\]
Consequently, if $T$ is the tangent line to a regular piece of $S$,
then the phase trace is
\begin{equation*}\label{eq:2d-positive-trace}
 \psi=\frac{\pi}{2}+\phase_1(D_T^2v)
 \quad\text{on the regular pieces of }\{x_1=0\}\cup\{\rho=R\},
\end{equation*}
and
\begin{equation*}\label{eq:2d-negative-trace}
 \psi=-\frac{\pi}{2}+\phase_1(D_T^2v)
 \quad\text{on the regular pieces of }\{x_2=0\}\cup\{x_2=\pm R\}.
\end{equation*}
These limits are independent of the side from which the singular set is approached.

It remains to consider the intersections of the singular lines.  At
the origin, \eqref{eq:v11} and \eqref{eq:v22} imply
\[
 \lambda_1(D^2v)\longrightarrow+\infty,
 \qquad
 \lambda_2(D^2v)\longrightarrow-\infty,
\]
so $\psi\to0$.

At $(R,0)$, we again have $\lambda_1(D^2v)\to+\infty$.  If $x_2>0$, take
$e=\frac {1}{\sqrt 2}(1,-1)$; the positive singular term generated by
$|\rho-R|^\alpha$ cancels in this direction, and
\[
 \langle e,D^2v\,e\rangle
 =-\frac{\alpha(\alpha-1)}{2}|x_2|^{\alpha-2}+O(1)
 \longrightarrow-\infty.
\]
If $x_2<0$, the same conclusion follows with $e=\frac{1}{\sqrt2}(1,1)$. Finally, $\lambda_2 (D^2v) = -\infty$ along $\{x_2 =0\}$. Thus
$\lambda_2(D^2v)\to-\infty$ and $\psi\to0$ at $(R,0)$.  Reflection gives the
same limits at $(-R,0)$.

We finally treat $(0,R)$; the other remaining junction follows by symmetry.
Write $x_2=R+t$, and put
$c_\alpha=\alpha(\alpha-1)$ and $c_\beta=\beta(\beta-1)$.  For $x_1>0$,
\[
D^2v=
\begin{pmatrix}
 c_\alpha\big(x_1^{\alpha-2}+|x_1+t|^{\alpha-2}\big)
 &c_\alpha|x_1+t|^{\alpha-2}\\
 c_\alpha|x_1+t|^{\alpha-2}
 &c_\alpha|x_1+t|^{\alpha-2}
  -4c_\beta|t(t+2R)|^{\beta-2}(R+t)^2+O(1)
\end{pmatrix}.
\]
In particular, $\lambda_1(D^2v)\geq v_{11}\to+\infty$.  To control the lower
eigenvalue, fix $\eta\in(0,1)$.  In the region $t\geq-\eta x_1$, one has
$|x_1+t|\geq c_\eta|t|$ for a constant $c_\eta>0$.  Hence
\[
 v_{22}\leq C_\eta|t|^{\alpha-2}-c|t|^{\beta-2}+O(1)
 \longrightarrow-\infty,
\]
because $\beta<\alpha$.  The same estimate holds in the region
$t\leq-x_1/\eta$, since there $|x_1+t|\geq(1-\eta)|t|$.  In the remaining
region $-x_1/\eta<t<-\eta x_1$, we have $x_1\geq\eta|t|$.  Taking
$e=\frac{1}{\sqrt 2} (1,-1)$ yields
\begin{align*}
 \langle e,D^2v\,e\rangle
 &=\frac{c_\alpha}{2}x_1^{\alpha-2}
   -2c_\beta|t(t+2R)|^{\beta-2}(R+t)^2+O(1)\\
 &\leq \frac{c_\alpha}{2}(\eta|t|)^{\alpha-2}
   -2c_\beta|t(t+2R)|^{\beta-2}(R+t)^2+O(1)
 \longrightarrow-\infty.
\end{align*}
By evenness in $x_1$, $\lambda_2(D^2v)\to-\infty$ along every path to $(0,R)$, and so
$\psi\to0$.  Evenness in $x_2$ gives the same conclusion at
$(0,-R)$.

We have therefore obtained a unique phase limit at every point of $S$.  The
limits in \eqref{eq:2d-positive-trace} and \eqref{eq:2d-negative-trace} lie
strictly between $-\pi$ and $\pi$, while the limits at all singular
intersections are zero.  Hence $\psi$ extends continuously to $\overline B_1$
with values in $(-\pi,\pi)$.
\end{proof}

From now on, we shall refer to the continuous extension of $\psi$ to $\overline B_1$ also as $\psi$. Define $g: \partial B_1 \to \mathbb{R}$ by 
$$ g(x) = v(x) + | x_1| - 2 |x_2| .$$

\begin{theorem}\label{thm:2d} Fix
$$
1<\beta<\alpha<2,
\quad
R\in
\left(
0,
\left(\frac1\alpha\right)^{1/(\alpha-1)}
\right).$$
Then, for any 
$$
C\in
[
0,
\min (
4R,
1-\alpha R^{\alpha-1}
)
),
$$

$$
u(x)
=v(x)
+|x_1|
-2|x_2|
+C(R-\rho(x))^+ 
$$
is a viscosity solution of the Dirichlet problem
$$ \arctan \lambda_1(D^2 w)+ \arctan \lambda_2 (D^2 w ) = \psi (x) \text{ in } B_1, \quad w = g \text{ on } \partial B_1.$$ 
\end{theorem}

\begin{proof}
Let $S$ be the singular set of $v$ as defined in the previous lemma.  Away
from $S$,
\[
 D^2u=D^2v,
\]
since $u-v$ is affine on each connected component of $B_1\setminus S$.
Therefore $u$ is a classical solution of the equation on $B_1\setminus S$.
It remains to check the viscosity inequalities on $S$.

We first determine which test functions can occur.  Away from $S$, we have
\[
 u_1(x_1,x_2)=\sgn x_1\big(\alpha|x_1|^{\alpha-1}+1
 +\alpha|\rho(x)-R|^{\alpha-1}\sgn(\rho(x)-R)
 -C\1_{\{R>\rho\}}\big).
\]
Thus
\begin{align*}
 \lim_{x_1\downarrow0}u_1(x_1,x_2)
 &=1+\alpha\big||x_2|-R\big|^{\alpha-1}\sgn(|x_2|-R)
   -C\1_{\{R>|x_2|\}})\\
 &\geq1-\alpha R^{\alpha-1}-C>0,
\end{align*}
and
\begin{align*}
 \lim_{x_1\uparrow0}u_1(x_1,x_2)
 &=-\big(1+\alpha\big||x_2|-R\big|^{\alpha-1}\sgn(|x_2|-R)
   -C\1_{\{R>|x_2|\}}\big)\\
 &\leq-1+\alpha R^{\alpha-1}+C<0.
\end{align*}
Hence no $C^2$ function can touch $u$ from above at a point of
$\{x_1=0\}$.

Similarly,
\begin{align*}
 u_2(x_1,x_2)
 &=\sgn x_2\big(-\alpha|x_2|^{\alpha-1}-2
 +\alpha|\rho(x)-R|^{\alpha-1}\sgn(\rho(x)-R)
 -C\1_{\{R>\rho\}}\big)\\
 &\quad-2x_2\sgn(x_2^2-R^2)
       \big(1+\beta|x_2^2-R^2|^{\beta-1}\big).
\end{align*}
Consequently,
\begin{align*}
 \lim_{x_2\downarrow0}u_2(x_1,x_2)
 &=-2+\alpha\big||x_1|-R\big|^{\alpha-1}\sgn(|x_1|-R)
   -C\1_{\{R>|x_1|\}}\\
 &\leq-2+\alpha<0,
\end{align*}
whereas
\begin{align*}
 \lim_{x_2\uparrow0}u_2(x_1,x_2)
 &=2-\alpha\big||x_1|-R\big|^{\alpha-1}\sgn(|x_1|-R)
   +C\1_{\{R>|x_1|\}}\\
 &\geq2-\alpha>0.
\end{align*}
Thus no $C^2$ function can touch $u$ from below at a point of
$\{x_2=0\}$.

Similarly, at a regular point of $\{x_2=\pm R\}$, the term
$-|x_2^2-R^2|$ creates a strict downward corner in the $x_2$ direction;
all other terms have matching one-sided first derivatives there. Hence no
$C^2$ function can touch $u$ from below at such a point.  At $(0,R)$,
\[
 \lim_{x_2\downarrow R}u_2(0,x_2)
 =-\alpha R^{\alpha-1}-2-2R,
\]
while
\[
 \lim_{x_2\uparrow R}u_2(0,x_2)
 =-\alpha R^{\alpha-1}-2-C+2R.
\]
Their difference is $-4R+C<0$, so no lower test exists at $(0,R)$; by
symmetry, the same holds at $(0,-R)$.  Since these two points also lie on
$\{x_1=0\}$, neither kind of test occurs there.

At a regular point $(x_1, x_2)$ of $\{\rho=R\}$, vary $x_1$ transversely to
the interface while remaining in one orthant. Consider $h(t)= u ( x_1+t , x_2).$ Modulo a smooth function, the one-dimensional restriction has the form
\[
 h(t)=|t|^\alpha+C(- t \sgn x_1)^+
\]
near $t=0$.
This function has no upper $C^2$ test at zero. If $C>0$, the positive corner of $(-t \sgn x_1))^+$ rules out contact, and if $C=0$, this follows from the fact that $\alpha <2$. This also proves the absence of
an upper test at $(\pm R,0)$.  Since those points lie on $\{x_2=0\}$, neither
kind of test occurs there.  At $(0,0)$ the $x_1$-corner rules out upper tests
and the $x_2$-corner rules out lower tests.

It remains to verify the appropriate viscosity inequality at regular points of $S$.  Set
\[
 S_+=\{x_1=0\}\cup\{\rho=R\},
 \qquad
 S_-=\{x_2=0\}\cup\{x_2=R\}\cup\{x_2=-R\}.
\]
Let $x_0 \in S_+$ be a regular point of $S$, let $T$ be the tangent line to $S$ at $x_0$, and let
$\varphi\in C^2$ touch $u$ from below at $x_0$.  Restriction to $T$ gives
\[
 D_T^2\varphi(x_0)\leq D_T^2u(x_0)=D_T^2v(x_0).
\]
By Lemma~\ref{lem:compression}, monotonicity of $\phase_1$, and
\eqref{eq:2d-positive-trace},
\begin{align*}
 \phase_2(D^2\varphi(x_0))
 &\leq\frac{\pi}{2}+\phase_1(D_T^2\varphi(x_0))\\
 &\leq\frac{\pi}{2}+\phase_1(D_T^2v(x_0))
 =\psi(x_0).
\end{align*}
Thus $u$ is a viscosity supersolution at $x_0$.

Now let $x_0 \in S_-$ be a regular point of $S$, and let $\varphi\in C^2$ touch $u$
from above at $x_0$.  Then
\[
 D_T^2\varphi(x_0)\geq D_T^2u(x_0)=D_T^2v(x_0),
\]
and the other half of Lemma~\ref{lem:compression}, together with
\eqref{eq:2d-negative-trace}, gives
\begin{align*}
 \phase_2(D^2\varphi(x_0))
 &\geq-\frac{\pi}{2}+\phase_1(D_T^2\varphi(x_0))\\
 &\geq-\frac{\pi}{2}+\phase_1(D_T^2v(x_0))
 =\psi(x_0).
\end{align*}
Thus $u$ is a viscosity subsolution at $x_0$.  At every intersection of a
positive and a negative singular hypersurface, both types of test have already
been excluded, so both viscosity conditions are vacuous.  This proves the
equation throughout $B_1$.

Finally, since $\rho\geq1>R$ on $\partial B_1$, the term
$C(R-\rho(x))^+$ is compactly supported in $B_1$.  Hence $u=g$ on
$\partial B_1$, independently of $C$.
\end{proof}

The solutions furnished by Theorem~\ref{thm:2d} are pairwise distinct, since
for two parameter values $C\neq C'$ their values at the origin differ by
$(C-C')R$.

\section{Dimensions three and higher}\label{sec:higher-dim}

We now give a construction that works uniformly for every dimension $n\geq3$.  Fix
\begin{equation*}\label{eq:nd-parameters}
 n\geq3,
 \qquad 1<\alpha<2,
 \qquad 0<R<1, \qquad m \in \{ 1, \dots, n-2\}
\end{equation*}
and set
\begin{equation*}\label{eq:nd-rho-v}
 \rho(x)=\sum_{j=1}^n|x_j|,
 \qquad
 v(x)=\sum_{j=1}^m|x_j|^\alpha-\sum_{j=m+1}^n|x_j|^\alpha+|\rho(x)-R|^\alpha.
\end{equation*}
Let
\begin{equation*}\label{eq:nd-S}
 S=\bigcup_{j=1}^n\{x_j=0\}\cup\{\rho=R\}.
\end{equation*}
On $\overline{B_1}\setminus S$, define
\begin{equation}\label{eq:nd-phase-def}
 \psi(x)=\phase_n(D^2v(x)).
\end{equation}

\begin{proposition}\label{prop:nd-phase}
The phase in \eqref{eq:nd-phase-def} extends uniquely to a function
\[
 \psi\in C\left(\overline{B_1};\left(-n\frac{\pi}2,n\frac{\pi}2\right)\right).
\]
Moreover,
\begin{equation*}\label{eq:nd-origin-phase}
 \psi(0)=-(n-2m)\frac{\pi}{2}.
\end{equation*}
\end{proposition}

\begin{proof}
Let $p=2-\alpha\in(0,1)$ and $\cA=\alpha(\alpha-1)$.  On a fixed orthant,
write $\sigma_j=\sgn x_j$.  Then
\begin{equation*}\label{eq:nd-hessian}
 D^2v=D+b\,\sigma\otimes\sigma,
\end{equation*}
where
\begin{equation*}\label{eq:nd-d-b}
 D=\operatorname{diag}(d_1,\dots,d_n),
 \quad d_j=\begin{cases} \cA |x_j|^{-p} & j \leq m \\-\cA|x_j|^{-p}\ & j\geq m+1, \end{cases}
 \quad b=\cA|\rho-R|^{-p}.
\end{equation*}
By Lemma \ref{lem:rank-one-formula},
\begin{equation}\label{eq:nd-phase-formula}
 \psi(x)=\sum_{j=1}^n\arctan d_j+
 \Arg(1+bQ+i bP),
\end{equation}
with
\begin{equation*}\label{eq:PQ}
 P=\sum_{j=1}^n\frac{1}{1+d_j^2},
 \qquad
 Q=\sum_{j=1}^n\frac{d_j}{1+d_j^2}.
\end{equation*}
In particular, the right-hand side is independent of the signs $\sigma_j$.

Suppose first that $y \in S \setminus \{\rho=R\}$.  As $x_j\to0$, the corresponding
$d_j$ tends to $+\infty$ for $j \leq m $ and to $-\infty$ for $j\geq m+1$, while its
contributions to $P$ and $Q$ tend to zero.  Thus every term in
\eqref{eq:nd-phase-formula} has a limit.  If all coordinates tend to zero,
then $P,Q\to0$ and the last argument tends to zero because $b$ remains finite.
If at least one coordinate is nonzero, then the limiting $P$ is positive and
$1 + bQ + i b P $ remains in the upper half-plane.  This gives a continuous
extension across all coordinate strata away from $\{\rho=R\}$.

Now let $y \in\{\rho=R\}$.  Rewrite the last term as
\[
 \Arg(b^{-1}+Q+iP).
\]
At least one coordinate of $y$ is nonzero because $R>0$, and hence the
limiting value of $P$ is strictly positive.  Since $b^{-1}\to0$, this argument
also has a limit as $x \to y$, including at all intersections with
coordinate hyperplanes. Formula \eqref{eq:nd-phase-formula} therefore defines
a continuous extension to $\overline{B_1}$.

For the range, observe that $d_j>0$ for $j \leq m$, $d_j<0$ for $j\geq m+1 $, and the argument term
in \eqref{eq:nd-phase-formula} belongs to $[0,\pi]$ in the limiting
formula. Hence
\[
 \psi> -(n-m)\frac{\pi}{2}>-n\frac{\pi}{2},
 \qquad
 \psi<(m+2)\frac{\pi}{2}\leq n\frac{\pi}{2}.
\]
The upper inequality is strict because the negative diagonal
terms contribute a strictly negative amount.  Finally, at the origin all
$d_j$ diverge, $b$ remains finite, and $P,Q\to0$.  Thus
\[
 \psi(0)=m\frac{\pi}{2}-(n-m)\frac{\pi}{2}
 =-(n-2m)\frac{\pi}{2}. \qedhere
\]
\end{proof}

We next identify the phase trace on the strata where a viscosity test may
occur.  Given $x \in S$, set
\[
 Z(x)=\{j \: :\: x_j=0\}.
\]
We call $\{x_j=0\}$ with $j \leq m$ and $\{\rho=R\}$ the positive singular hypersurfaces,
and $\{x_j=0\}$, $j\geq m+1$, the negative singular hypersurfaces.  A stratum is
\emph{pure positive} if it lies in at least one positive hypersurface and in no
negative one; it is \emph{pure negative} if it lies in at least one negative
hypersurface and in no positive one.  All remaining singular strata are mixed.
On a pure stratum, let $T$ denote its tangent space.

\begin{lemma}\label{lem:phase-trace}
Let $x\in S$.

\begin{enumerate}
\item If $x$ lies on a pure positive stratum of codimension $k\in\{|Z(x)|, |Z(x)|+1\}$,
then
\begin{equation}\label{eq:positive-trace}
 \psi(x)=k\frac{\pi}{2}+\phase_{n-k}\big(D_T^2v(x)\big).
\end{equation}
\item If $x$ lies on a pure negative stratum of codimension
$k=|Z(x)|$, then
\begin{equation}\label{eq:negative-trace}
 \psi(x)=-k\frac{\pi}{2}+\phase_{n-k}\big(D_T^2v(x)\big).
\end{equation}
\end{enumerate}
\end{lemma}

\begin{proof}
On a pure negative stratum, $b$ and all nonsingular $d_j$ have finite limits,
whereas precisely $k$ diagonal entries tend to $-\infty$. In \eqref{eq:nd-phase-formula}, the contributions of these $k$ singular diagonal entries to $P$ and $Q$ in go to zero, and they contribute
$-k\frac{\pi}{2}$ to the first sum; what remains is exactly the phase of the
tangential Hessian. This proves \eqref{eq:negative-trace}.

On a pure positive stratum away from $\{\rho=R\}$, the same
argument gives $k$ contributions of $+\frac{\pi}{2}$ while again the singular terms contribute nothing to $P$ and $Q$. On $\{\rho=R\}$ away from all
coordinate hyperplanes, \eqref{eq:positive-trace} follows from Lemma
\ref{lem:rank-one-limit}. Next, consider a pure positive stratum which lies on the set $\{ \rho =R\}$. All coordinates $x_{m+1},\dots,x_n$ are nonzero. After relabeling the coordinates, if necessary, suppose that the first $k-1$ coordinates of $x$ are zero on this stratum ($k-1 \leq m)$. Divide by $b$ in
the second term of \eqref{eq:nd-phase-formula} and write it as
$\Arg(b^{-1}+Q+iP)$. For $j\leq k-1$, the $d_j$ contributions to $P,Q$ vanish at $x$, and $b^{-1} \to 0$ at $x$, so the argument term converges to
$\Arg(Q'+iP')$, where $P',Q'$ are formed from the remaining coordinates $x_j$, $j\geq k$. Put
\[
 D'=\operatorname{diag} (d_{k} , \dots, d_n),
 \qquad \sigma'=(\sigma_{k}, \dots, \sigma_n).
\]
By Lemmas~\ref{lem:rank-one-formula} and \ref{lem:rank-one-limit}, after normalizing $\sigma'$,
\begin{align*}
 \sum_{j =k}^n \arctan d_j+\Arg(Q'+iP')
 &=\lim_{t\to+\infty}\phase_{n-(k-1)}(D'+t\,\sigma'\otimes\sigma')\\
 &=\frac{\pi}{2}
   +\phase_{n-k}\big(D'|_{(\sigma')^\perp}\big).
\end{align*}
The first sum in \eqref{eq:nd-phase-formula} also contributes $(k-1)\frac{\pi}{2}$ from
the singular $d_j$, and
$\{0\}^{k-1}\times(\sigma')^\perp$ is exactly the tangent space of the stratum. This proves \eqref{eq:positive-trace}.
\end{proof}

Let $L=4n$.  For $C\in[0,1]$, define
\begin{equation}\label{eq:nd-family}
 u_C(x)=v(x)+L \sum_{j=1}^m|x_j|-L\sum_{j=m+1}^n|x_j|+C(R-\rho(x))^+,
\end{equation}
and define the common boundary value
\begin{equation}\label{eq:nd-boundary}
 g(x)=v(x)+L \sum_{j=1}^m|x_j|-L\sum_{j=m+1}^n|x_j|,
 \qquad\text{on }\partial B_1.
\end{equation}
Again $\rho\geq1>R$ on $\partial B_1$, so the tent term vanishes there.

\begin{theorem}\label{thm:nd}
For every $n\geq3$ and every $C\in[0,1]$, the function $u_C$ in
\eqref{eq:nd-family} is a viscosity solution of
\begin{equation*}\label{eq:nd-DP}
 \phase_n(D^2u)=\psi(x)\quad\text{in }B_1,
 \qquad u=g\quad\text{on }\partial B_1.
\end{equation*}
The family $\{u_C:0\leq C\leq1\}$ consists of pairwise distinct solutions.
\end{theorem}

\begin{proof}
On every connected component of $B_1\setminus S$, the function $u_C-v$ is
affine.  Hence $D^2u_C=D^2v$ there, and $u_C$ solves the equation classically.
We analyze points of $S$.

\smallskip
\noindent\emph{Step 1: exclusion of test functions.}
Suppose $x_{0,j} =0$ for some $j \leq m$. Varying $x_j=t$, while keeping all other coordinates fixed, the coefficient of $|t|$ in
the one-dimensional expansion of $u_C$ is at least
\[
 L-\alpha R^{\alpha-1}-C\geq4n-3>0.
\]
Thus the graph has a strict upward corner in the $e_j$ direction, and no
$C^2$ function can touch $u_C$ from above at $x_0$.

Next suppose $x_{0,j}=0$ for some $j\geq m+1$.  Varying $x_j=t$, the coefficient
of $|t|$ away from $\rho=R$ is
\[
 -L+\alpha|\rho(x_0)-R|^{\alpha-1}\sgn(\rho(x_0)-R)-C\1_{\{s<R\}}.
\]
Since $\rho(x_0)\leq\sqrt n$ in $B_1$,
\[
 -L+\alpha((\rho(x_0)-R)^+)^{\alpha-1}
 \leq-4n+2\sqrt n<0.
\]
If $\rho(x_0)=R$, the terms $-|t|^\alpha$ and $|\rho(x_0)+|t|-R|^\alpha$ cancel and the
leading corner is still $-L|t|$.  Thus no $C^2$ function can touch $u_C$ from
below at a point of any negative coordinate hyperplane.

Finally suppose $\rho(x_0)=R$. If $x_0$ is not already covered by a positive coordinate hyperplane, choose a nonzero coordinate, say $x_1=t$, and vary it transversely to the
interface while staying in one orthant. 
Up to a smooth function, the one-dimensional
restriction is
\[
 |t|^\alpha+C(-t \sgn x_{0,1})^+,
\]
This function admits no upper $C^2$ test at zero.  Indeed, if
$q(t)=pt+O(t^2)$ touched it from above, the inequality for $t>0$ would give
$p\geq0$, while the inequality for $t=-s<0$ would give $p\leq-C$.  This is
impossible if $C>0$; if $C=0$, then $p=0$, and $O(t^2)$ cannot dominate
$a|t|^\alpha$ because $\alpha<2$. Therefore every positive singular
hypersurface rules out tests from above, and every negative singular
hypersurface rules out tests from below.
At a mixed stratum both kinds of test are absent, so both viscosity conditions
are vacuous.

\smallskip
\noindent\emph{Step 2: pure positive strata.}
Let $x_0$ lie on a pure positive stratum of codimension $k$, and let
$\varphi\in C^2$ touch $u_C$ from below at $x_0$.  The restriction of
$\varphi$ to the stratum touches the smooth restriction of $u_C$ from below.
Since the added absolute-value and tent terms are affine when restricted to a
fixed stratum,
\[
 D_T^2\varphi(x_0)\leq D_T^2u_C(x_0)=D_T^2v(x_0).
\]
By Lemma \ref{lem:compression}, monotonicity of $\phase_{n-k}$, and Lemma
\ref{lem:phase-trace},
\begin{align*}
 \phase_n(D^2\varphi(x_0))
 &\leq\frac{k\pi}{2}+\phase_{n-k}(D_T^2\varphi(x_0))\\
 &\leq\frac{k\pi}{2}+\phase_{n-k}(D_T^2v(x_0))
 =\psi(x_0).
\end{align*}
This is precisely the supersolution inequality.

\smallskip
\noindent\emph{Step 3: pure negative strata.}
Let $x_0$ lie on a pure negative stratum of codimension $k$, and let
$\varphi$ touch $u_C$ from above.  Restricting to the stratum gives
\[
 D_T^2\varphi(x_0)\geq D_T^2u_C(x_0)=D_T^2v(x_0).
\]
The other half of Lemma \ref{lem:compression} and
\eqref{eq:negative-trace} imply
\begin{align*}
 \phase_n(D^2\varphi(x_0))
 &\geq-k\frac{\pi}{2}+\phase_{n-k}(D_T^2\varphi(x_0))\\
 &\geq-k\frac{\pi}{2}+\phase_{n-k}(D_T^2v(x_0))
 =\psi(x_0).
\end{align*}
Thus $u_C$ is a subsolution.  Together with Steps 1-3, this proves that
$u_C$ is a viscosity solution throughout $B_1$.

The boundary condition follows from \eqref{eq:nd-boundary}.  The solutions are
distinct because
\[
 u_C(0)=R^\alpha+CR. \qedhere
\]
\end{proof}

\noindent \textbf{Remark.}
The construction above produces phase functions which achieve the special phase values $-(n-2m) \frac{\pi}{2}$ for $1 \leq m \leq n-2$, but the positive critical value $(n-2)\frac{\pi}{2}$ is missing. The construction above does not work for the top value without modification, since with $m=n-1$, the intersection $\bigcap_{j=1}^m \{x_j =0 \} \cap \{ \rho =R\}$ is a pure positive stratum of codimension $n$, so the phase $\psi$ would achieve the value $n\frac{\pi}{2}$. This violates the requirement that $\psi \in C ( \overline {B_1} ; (-n\frac{\pi}{2},n\frac{\pi}{2}))$. However, the positive critical value of the phase can be achieved by taking $m=1$ and replacing  $v, \psi, u_C$, and $g$ with the opposite sign.

\smallskip 

We finish by verifying that the continuous phase in the higher-dimensional
construction genuinely crosses the special value.

\begin{proposition}\label{prop:critical-crossing}
For the phase of Proposition \ref{prop:nd-phase}, there exists $t_0>0$ such
that
\[
 \psi(te_1)<-(n-2m)\frac{\pi}{2}<\psi(te_n)
 \qquad\text{for }0<t<t_0.
\]
More precisely, if $p=2-\alpha$ and $\cA=\alpha(\alpha-1)$, then
\begin{align}
 \psi(te_1)&=-(n-2m)\frac{\pi}{2}-\frac{t^p}{\cA}+O(t^{2p}),
 \label{eq:e1-expansion}\\
 \psi(te_n)&=-(n-2m)\frac{\pi}{2}+\frac{t^p}{\cA}+O(t^{2p}).
 \label{eq:e2-expansion}
\end{align}
\end{proposition}

\begin{proof}
The values on the axes are understood through the continuous extension in
\eqref{eq:nd-phase-formula}.  Put $s=\cA t^{-p}$ and
$b_t=\cA(R-t)^{-p}$.  Along the $x_1$-axis, the terms $\arctan d_j$ contribute $+ \frac{\pi}{2}$ for $2 \leq j \leq m$ and $- \frac{\pi}{2} $ for $j \geq m+1$. Therefore
\[
 \psi(te_1)=\arctan s+(m-1)\frac{\pi}{2}-(n-m) \frac \pi 2
 +\Arg\left(1+b_t\frac{s}{1+s^2}
              +i b_t\frac{1}{1+s^2}\right).
\]
Since
\[
 \arctan s=\frac{\pi}{2}-\frac1s+O(s^{-3}),
 \qquad
 \Arg\left(1+b_t\frac{s}{1+s^2}
              +i b_t\frac{1}{1+s^2}\right)=O(s^{-2}),
\]
we obtain \eqref{eq:e1-expansion}.  Along the $x_n$-axis, the finite entry is
$-s$, while the positive entries contribute $m \frac{\pi}{2}$ and the remaining negative
entries contribute $-(n-(m+1))\frac{\pi}{2}$. Hence
\[
 \psi(te_n)=-(n-2m)\frac{\pi}{2}+\frac1s+O(s^{-2}),
\]
which is \eqref{eq:e2-expansion}.  The asserted inequalities follow for small
$t$.
\end{proof}
We now present the proof of Theorem~\ref{thm:DP}.

\begin{proof}[Proof of Theorem~\ref{thm:DP}]
For $n=2$, the conclusion follows from Proposition \ref{lem:2d-phase} and Theorem~\ref{thm:2d}. For $n\geq3$ and $1\leq m\leq n-2$, it follows from Proposition~\ref{prop:nd-phase}, Theorem~\ref{thm:nd},  and Proposition~\ref{prop:critical-crossing}.

For $m=n-1$,  start with the construction for $m=1$, and let $\Pi$ be the permutation interchanging $e_1$ and $e_n$. Let
\[
 \widetilde v(x)=-v(\Pi x),\qquad
 \widetilde\psi(x)=-\psi(\Pi x),\qquad
 \widetilde u_C(x)=-u_C(\Pi x),\qquad
 \widetilde g(x)=-g(\Pi x).
\]
Since $\phase_n(-A)=-\phase_n(A)$, the functions $\widetilde u_C$ form a continuum of pairwise distinct viscosity solutions of
\[
 \phase_n(D^2\widetilde u_C)=\widetilde\psi(x)\quad\text{in }B_1,
 \qquad \widetilde u_C=\widetilde g\quad\text{on }\partial B_1.
\]
Observe that
\[
 \widetilde\psi(0)=(n-2)\frac{\pi}{2}
 =-(n-2(n-1))\frac{\pi}{2},
\]
and Proposition~\ref{prop:critical-crossing}, applied to the $m=1$ phase, gives
\[
 \widetilde\psi(te_1)=-\psi(te_n)
 <(n-2)\frac{\pi}{2}
 <-\psi(te_1)=\widetilde\psi(te_n)
\]
for all sufficiently small $t>0$. This proves the theorem.
\end{proof}

\bibliographystyle{amsalpha}
\bibliography{Library1}
\end{document}